\documentclass{amsart}

\makeatletter
\@namedef{subjclassname@2010}{%
 \textup{2010} Mathematics Subject Classification}
\makeatother
\usepackage[T1]{fontenc}
\usepackage{textcomp}

\usepackage{amssymb}

\usepackage{amsrefs}

\usepackage{calc}

\usepackage{xspace}

\usepackage{graphics}

\theoremstyle{plain}
\newtheorem{thm}{Theorem}[subsection]
\newtheorem*{thm*}{Theorem}
\newtheorem*{lem*}{Lemma}
\newtheorem{cor}[thm]{Corollary}
\newtheorem{prop}[thm]{Proposition}
\newtheorem{lem}[thm]{Lemma}

\theoremstyle{definition}
\newtheorem{definition}[thm]{Definition}

\newtheorem{rmk}[thm]{Remark}

\newcommand{\atd}{ancestor-to-descendant\space}
\newcommand{\ATD}{\mathrm{ATD}}
\newcommand{\btr}{\mathbb{T}}
\newcommand{\CLLD}{coarse co-Lipschitz}
\newcommand{\ctr}{\mathcal{T}}

\newcommand{\eps}{\varepsilon}
\newcommand{\ffrac}{\displaystyle \frac}
\newcommand{\lev}{\mathrm{level}}
\newcommand{\Lip}{\mathrm{Lip}}
\newcommand{\LLD}{coarse Lipschitz\space}
\newcommand{\mks}{\mathbb{M}_{\infty}}
\newcommand{\mS}{\mathcal{S}}
\newcommand{\mtr}{\mathbb{T}_{\infty}}
\newcommand{\ov}{\overline}

\newcommand{\rt}{\mathrm{root}}
\newcommand{\sep}{\mathrm{sep}}

\newcommand{\ttilde}{\widetilde}
\newcommand{\tto}{\longrightarrow}

\begin{document}

\title{The transfer of property $(\beta)$ of Rolewicz by a uniform quotient map}

\author{S.~J.~ Dilworth}
\author{Denka Kutzarova}
\author{N. Lovasoa Randrianarivony}

\address{Department of Mathematics\\
University of South Carolina\\
Columbia, SC 29208, USA.}
\email{dilworth@math.sc.edu}

\address{Department of Mathematics\\
University of Illinois at Urbana-Champaign\\
Urbana, IL 61801, USA.}
\email{denka@math.uiuc.edu}

\address{Department of Mathematics and Computer Science\\
Saint Louis University \\
St. Louis, MO  63103, USA.}
\email{nrandria@slu.edu}


\thanks{The first author was partially supported by NSF grant DMS--1101490.  The third author was partially supported by NSF grant DMS--1301591.}

\maketitle

\begin{abstract}
We provide a Laakso construction to prove that the property of having an equivalent norm with the property $(\beta)$ of Rolewicz is qualitatively preserved via surjective uniform quotient mappings between separable Banach spaces. On the other hand, we show that the $(\beta)$-modulus is not quantitatively preserved via such a map by exhibiting two uniformly homeomorphic Banach spaces that do not have $(\beta)$-moduli of the same power-type even under renorming.
\end{abstract}

\section{Introduction}
A map $f:X\tto Y$ between two metric spaces is called a uniform quotient mapping \cite[Chapter 11]{BenyaminiLindenstrauss2000} if there exist two nondecreasing maps $\omega,\Omega:[0,\infty)\tto[0,\infty)$ such that $\omega(r)>0$ for $r>0$, and $\Omega(r)\to0$ as $r\to 0$, and for every $x\in X$ and every $r>0$,
$$B(f(x),\omega(r))\subseteq f(B(x,r))\subseteq B(f(x),\Omega(r)).$$
(In this article, all balls will be considered closed unless otherwise specified.)  If there are constants $c,L>0$ such that one can use $\omega(r)=cr$ and $\Omega(r)=Lr$, then the map $f$ is called a Lipschitz quotient mapping.  If the map $f$ is surjective, the metric space $Y$ is called a uniform quotient (resp. Lipschitz quotient) of the metric space $X$.

\medskip

Closely related to the above definitions are the notions of coarse Lipschitz maps and coarse co-Lipschitz maps.  A map $f:X\to Y$ between two metric spaces is called Lipschitz for large distances, or coarse Lipschitz, if for every ${\delta}>0$ there exists $L({\delta})>0$ such that $d_Y(f(x),f(x'))\leq L({\delta})\, d_X(x,x')$ for every $x,x'\in X$ with $d_X(x,x')\geq {\delta}$.  The map $f$ is called co-Lipschitz for large distances, or coarse co-Lipschitz, if for every ${\delta}>0$ there exists $c({\delta})>0$ such that for every $R\geq {\delta}$, for every $y,y'\in Y$ with $d_Y(y,y')\leq c({\delta})R$, for every $x\in f^{-1}(y)$, there exists $x'\in f^{-1}(y')$ with $d_X(x,x')\leq R$.  
%for every $x\in X$, $B(f(x),c(\delta)\,R)\subseteq f(B(x,R))$.  
It is known that when the metric spaces $X$ and $Y$ are metrically convex (for example if $X$ and $Y$ are Banach spaces), then a surjective uniform quotient mapping between $X$ and $Y$ is always coarse Lispchitz and coarse co-Lipschitz (see \cite[Chapter 11]{BenyaminiLindenstrauss2000} and \cite[Section 3]{LimaRandrianarivony2010} for example).

\medskip

In \cite{LimaRandrianarivony2010} and in \cite{DilworthKutzarovaLancienRandrianarivony2012} the geometric property $(\beta)$ of Rolewicz is applied to the study of the rigidity of Banach spaces under uniform quotient mappings.  Rolewicz introduced property $(\beta)$ as an isometric generalization of uniform convexity in the sense that a uniformly convex Banach space necessarily has property $(\beta)$ (see \cite{Rolewicz1986} and \cite{Rolewicz1987}).  Furthermore a result of Baudier, Kalton, and Lancien \cite{BaudierKaltonLancien2010} implies that a separable reflexive Banach space $X$ can be renormed to have property $(\beta)$ if and only if the infinitely branching infinitely deep tree with its hyperbolic metric cannot bi-Lipschitz embed into $X$.  Contrasted with Bourgain's metric characterization of superreflexivity (see \cite{Bourgain1986} and \cite{Baudier2007}), this indicates that for separable spaces the isomorphic class of all Banach spaces with property $(\beta)$ is also larger than the isomorphic class of all uniformly convex Banach spaces.  We remark in passing that this statement is actually known to be true, see for example \cite{Kutzarova1989}.

\medskip

For the purpose of this article, the equivalent definition of property $(\beta)$ given in \cite{Kutzarova1991} is used.  A Banach space $X$ has property $(\beta)$ if for any $t\in (0,a]$, where the number $1\leq a\leq 2$ depends on the space X,  there exists $\delta>0$ such that for any element $x$ and sequence $(x_n)_{n\geq 1}$ in the unit ball $B_X$, if $\sep\left((x_n)_{n\geq 1}\right)=\inf\{\|x_n-x_m\|: n\neq m; n,m\geq 1\}\geq t$, then there exists $n\ge 1$ satisfying
$$\frac{\|x-x_n\|}{2}\le 1-\delta.$$
The $(\beta)$-modulus was defined in \cite{AyerbeDominguez_BenavidesCutillas1994} (although with a different notation) as
$$
\ov{\beta}_X(t) = 1 - \sup \left\{ \inf \left\{
\frac{\|x-x_n\|}{2} : n\geq 1 \right\} : (x_n)_{n\geq1} \subseteq B_X, x\in
B_X, \sep\left( (x_n)_{n\geq 1} \right) \ge t \right\}.
$$
%where $\sep\left((x_n)_{n\geq 1}\right)=\inf\{\|x_n-x_m\|: n\neq m, n,m\geq 1\}$.

\medskip

In \cite{DilworthKutzarovaLancienRandrianarivony2012} it is proven that if a Banach space $Y$ is a uniform quotient of another Banach space $X$ with property $(\beta)$, then one can compare the $(\beta)$-modulus of $X$ with the modulus of another asymptotic geometry on $Y$, namely asymptotic uniform smoothness, and under additional assumptions, asymptotic uniform convexity.  (See \cite{Milman1971} and \cite{JohnsonLindenstraussPreissSchechtman2002} for the notion of asymptotic uniform smoothness and asymptotic uniform convexity.  See also Section \ref{Denka's notes}.)  On the other hand, the authors in \cite{MendelNaor2010} use the (local) geometric property called Markov convexity to study the rigidity of Banach spaces under Lipschitz quotient mappings.  They do this by comparing the Markov convexity of a space $Y$ with the Markov convexity of a space $X$ when $Y$ is a Lipschitz quotient of a subset of $X$.  

\medskip

In light of these two approaches, the present article studies the asymptotic property $(\beta)$ of Rolewicz on a Banach space $Y$ that is a uniform quotient of another Banach space $X$ with property $(\beta)$.  This is done both in a qualitative sense and a quantitative sense.  We show that the property of having an equivalent norm with property $(\beta)$ is qualitatively preserved under a surjective uniform quotient mapping between separable Banach spaces.  We do this by first assuming reflexivity on the range space in Section \ref{first note}.  In Section \ref{second note}, we provide a modification of this approach to remove the reflexivity assumption.

\medskip

In contrast, we show in Section \ref{Denka's notes} that quantitatively the $(\beta)$-modulus is not preserved under a surjective uniform quotient mapping.  We also show that in full generality we cannot compare the $(\beta)$-modulus of the domain space with the modulus of asymptotic uniform convexity of the range space of a surjective uniform quotient map.

\medskip

\section{Case when the range space is reflexive}\label{first note}

One of the main ingredients in this section is a metric characterization of being isomorphic to a space with property $(\beta)$.  By combining the results of \cite[Theorem 6.3]{DilworthKutzarovaLancienRandrianarivony2012} and \cite[Theorem 1.2]{BaudierKaltonLancien2010}, we know that a separable reflexive Banach space $Y$ has an equivalent norm with property $(\beta)$ if and only if it does not contain any bi-Lipschitz copy of the metric tree $\mathbb{T}_{\infty}$ of all finite subsets of $\mathbb{N}$ with the shortest path metric.

\begin{thm}\label{thm1}
Let $X$ be a Banach space with (an equivalent norm with) property $(\beta)$.  Let $Y$ be a separable Banach space that is a uniform quotient of $X$.  If $Y$ is reflexive, then $Y$ has an equivalent norm with property $(\beta)$.
\end{thm}
	\begin{proof}
	Let $f:X\tto Y$ be a surjective uniform quotient mapping.  Then $f$ is coarse Lipschitz and coarse co-Lipschitz.  For a contradiction, assume that $Y$ does not admit any equivalent norm with property $(\beta)$.  Denote by $\Sigma$ a subset of $Y$ that is bi-Lipschitz equivalent to $\mathbb{T}_{\infty}$, with a bi-Lipschitz equivalence denoted by $j:\Sigma\tto \mathbb{T}_{\infty}$.  The restriction $f_{|_{f^{-1}(\Sigma)}}:f^{-1}(\Sigma)\tto \Sigma$ is still a uniform quotient map that is \LLD and \CLLD.  As a result, since $\mtr$ is discrete, the composition $j\circ f_{|_{f^{-1}(\Sigma)}}:f^{-1}(\Sigma)\tto \mathbb{T}_{\infty}$ is a (surjective) Lipschitz quotient map.  The theorem then follows from the next proposition. 
	
	\end{proof}

\begin{prop}\label{mainprop}
  It is not possible that $\mathbb{T}_{\infty}$ be a Lipschitz quotient of a subset of a Banach space with property $(\beta)$.
\end{prop}
	\begin{proof}
	We define a graph $\mathbb{M}_{\infty}$ considered as a metric space, and we define a map $\phi:\mathbb{T}_{\infty} \tto \mathbb{M}_{\infty}$.  This map $\phi$, although not a Lipschitz quotient map, will have a property still strong enough to draw a contradiction from the composition $\lambda=\phi\circ j\circ f_{|_{f^{-1}(\Sigma)}}:f^{-1}(\Sigma) \tto \mathbb{M}_{\infty}$, where $j\circ f_{|_{f^{-1}(\Sigma)}}:f^{-1}(\Sigma)\tto \mathbb{T}_{\infty}$ is a Lipschitz quotient map from a subset $f^{-1}(\Sigma)$ of a Banach space with property $(\beta)$.  We support all these claims in the upcoming subsections.
	\end{proof}

\subsection{Definition of the graph $\mathbb{M}_{\infty}$}\label{laakso}
\mbox{}

The graph $\mathbb{M}_{\infty}$ is defined after a variant of the Laakso construction (see for example \cite{Laakso2002, MendelNaor2010, CheegerKleiner2011}).  We define a sequence of graphs $(G_n)_{n\geq 1}$ such that $G_n\subseteq G_{n+1}$ for every $n\geq 1$ as follows.  The graph $G_1$ is the following ordered graph.  We start with a root.  The root has one immediate descendant.  This first descendant has countably infinitely many immediate descendants.  Each one of those second generation vertices has only one immediate descendant, which is a common immediate descendant to all of them.  The graph $G_1$ has diameter $3$.

	\begin{figure}[h] %  figure placement: here, top, bottom, or page
	   \centering
	   \scalebox{0.1}{\rotatebox{270.0}{\includegraphics{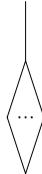}}}
	   \caption{The graph $G_1$.  The ``$\cdots$'' represents the fact that $G_1$ has infinitely many vertices at level $2$.}
	   %\label{fig:example}
	\end{figure}

To construct $G_{n+1}$ ($n\geq 1$), we take a copy of $G_1$ and we rescale it so each edge has length the diameter of $G_n$.  Next we replace each rescaled edge with a copy of $G_n$, matching the parent vertex of the rescaled edge with the oldest ancestor in the copy of $G_n$, and the child vertex of the rescaled edge with the youngest vertex in the copy of $G_n$.  This gives all the vertices of $G_{n+1}$.  Next we order the vertices of $G_{n+1}$ by going downward from the root all the way to the youngest generation, which is now at distance $3^{n+1}$ from the root.  Finally, we identify the root and the first $3^n$ generations of $G_{n+1}$ with (the root and the first $3^n$ generations of) $G_n$.  This way we have $G_n\subseteq G_{n+1}$.

	\begin{figure}[h] %  figure placement: here, top, bottom, or page
	   \centering
	   \scalebox{0.45}{\includegraphics{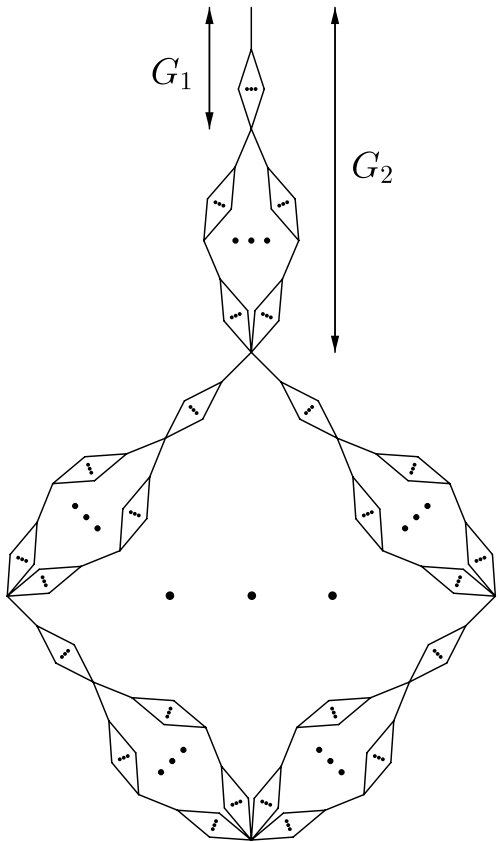}}
	   \caption{The graph $G_3$ with the subsets $G_1$ and $G_2$ highlighted.}
	   %\label{fig:example}
	\end{figure}

The graph $\mathbb{M}_{\infty}$ is defined to be the union $\displaystyle \bigcup_{n=1}^{\infty}G_n$, considered as a metric space with the shortest path metric.  The generational order of a given element of $\mathbb{M}_{\infty}$ will be called its {\sl level}.  Note that the level of an element is also the length of a shortest path from the root to that element.  We say that an element $\nu$ is a descendant of another element $\mu$, or that $\mu$ is an ancestor of $\nu$, both denoted $\mu<\nu$, if there is a shortest path from the root to $\nu$ that passes through $\mu$.

An element of $\mathbb{M}_{\infty}$ will be called {\sl non-branching} if it only has one immediate descendant; and it will be called {\sl branching} if it has infinitely many immediate descendants.  For each branching element of $\mathbb{M}_{\infty}$, we give a fraternal order among its immediate descendants by fixing a bijection between these descendants and the set $\mathbb{N}$ of the natural numbers.

\begin{rmk}
Laakso constructions, introduced in \cite{Laakso2000}, have appeared elsewhere as a tool to investigate the geometry of metric spaces, including that of Banach spaces.  The following list is not exhaustive: \cite{Tyson2005}, \cite{JohnsonSchechtman2009},  \cite{MendelNaor2010}, \cite{CheegerKleiner2011}, and recently \cite{Ostrovskii2013}.  We would like to point out the difference between our Laakso construction and that of \cite[Section 3]{MendelNaor2010} and \cite[Example 1.2]{CheegerKleiner2011} for example.  First, our basic block $G_1$ has diameter three instead of four.  Second, in the construction of our basic block $G_1$, the first generation vertex has infinitely many immediate descendants instead of just two as in the standard Laakso construction.  Third, when constructing $G_{n+1}$ from $G_n$, we do {\sl not} rescale down the copies of $G_n$.  In our case, $G_{n+1}$ has diameter three times as large as the diameter of $G_n$.  We also note that even if the Laakso construction in \cite{Laakso2002} can be made so that the basic block has diameter three, our construction is still different from the construction in \cite{Laakso2002}.

On the other hand we would like to mention that at the inception of this paper, our Laakso construction did have a basic block with diameter four.  The proof went all the way through.  However, as we prepared the final version of this article we realized that we did not need the last edge of the basic block, which brought us to the construction of $\mathbb{M}_{\infty}$ as it is presented here.

\end{rmk}

%\medskip

\subsection{Definition and properties of the map $\phi:\mathbb{T}_{\infty}\tto \mks$}{\label{phi}}
\mbox{}

We define $\phi(\rt)=\rt$.  By induction, if $\phi(\ov{n})$ is already defined for an element $\ov{n}$ of $\mathbb{T}_{\infty}$, and if $\ov{m}$ is an immediate descendant of $\ov{n}$ in $\mathbb{T}_{\infty}$, then
	\begin{itemize}
	\item set $\phi(\ov{m})$ to be the one immediate descendant of $\phi(\ov{n})$ if $\phi(\ov{n})$ is non-branching;
	\item otherwise, set $\phi(\ov{m})$ to be the immediate descendant of $\phi(\ov{n})$ in the same fraternal order among all immediate descendants of $\phi(\ov{n})$ as $\ov{m}$ is among all immediate descendants of $\ov{n}$.
	\end{itemize}

Before we list the properties of the map $\phi$, let us make the following definitions.

\begin{definition}
Let $f:X\to Y$ be a map between two metric spaces $X$ and $Y$.  Let $\mathcal{Q}$ be a subset of $X\times X$ and $\mathcal{R}$ a subset of $Y\times Y$.  We can, and will, think of $\mathcal{Q}$ and $\mathcal{R}$ as relations on $X$ and $Y$ respectively.  We say that the map $f:X\to Y$ is 
	\begin{itemize}
	\item $\mathcal{Q}$-Lipschitz if there exists $L^{\mathcal{Q}}>0$ such that for every $x, x'\in X$, if $x\mathcal{Q}x'$ then $d_Y(f(x),f(x'))\leq L^{\mathcal{Q}}\,d_X(x,x')$;
	\item $\mathcal{Q}$-Lipschitz for large distances, or $\mathcal{Q}$-coarse Lipschitz, if for every ${\delta}>0$ there exists $L^{\mathcal{Q}}({\delta})>0$ such that $d_Y(f(x),f(x'))\leq L^{\mathcal{Q}}({\delta})\,d_X(x,x')$ for every $x,x'\in X$ with $x\mathcal{Q}x'$ and with $d_X(x,x')\geq {\delta}$;
	\item $\mathcal{R}$-co-Lipschitz if there exists $c^{\mathcal{R}}>0$ such that for every $y,y'\in Y$ with $y\mathcal{R}y'$ and with $d_Y(y,y')\leq c^{\mathcal{R}}\,R$, for every $x\in f^{-1}(y)$, there exists $x'\in f^{-1}(y')$ with $d_X(x,x')\leq R$;
	
	\item $\mathcal{R}$-co-Lipschitz for large distances, or $\mathcal{R}$-coarse co-Lipschitz, if for every ${\delta}>0$ there exists $c^{\mathcal{R}}({\delta})>0$ such that for every $R\geq {\delta}$, for every $y,y'\in Y$ with $y\mathcal{R}y'$ and with $d_Y(y,y')\leq c^{\mathcal{R}}({\delta})\,R$, for every $x\in f^{-1}(y)$, there exists $x'\in f^{-1}(y')$ with $d_X(x,x')\leq R$.
	\end{itemize}
\end{definition}
In this article, the main relation we consider is that of ancestor-to-descendant.  %, where for two elements $x,y\in\mks$, $x<y$ if $x$ is an ancestor of $y$ and (equivalently) $y$ is a descendant of $x$.  
We now list the properties of the map $\phi$ in the following lemma.
	
\begin{lem}{\label{ppties of phi}}
\mbox{}
	\begin{enumerate}
	\item $\phi$ is surjective, and preserves levels.
	\item $\phi$ is $1$-Lipschitz.
	\item $\phi$ is ancestor-to-descendant $1$-co-Lipschitz in the following sense:
	$$\forall \mu,\nu\in\mks \text{ with } \mu<\nu, ~\forall \ov{n}\in \phi^{-1}(\mu), ~\exists \ov{m}\in \phi^{-1}(\nu) \text{ such that } d_{\mathbb{T}_{\infty}}(\ov{n},\ov{m})=d_{\mks}(\mu,\nu).$$
	\end{enumerate}

\end{lem}
	\begin{proof}
	\mbox{ }
	\begin{enumerate}
	\item This claim is easy.
	\item Let $\ov{n},\ov{m}\in\mathbb{T}_{\infty}$.  We consider two cases.
		\begin{itemize}
		\item Case $\ov{n}<\ov{m}$:
		
			\begin{equation*}
			\begin{split}
			d_{\mks}(\phi(\ov{n}),\phi(\ov{m})) &=\lev(\phi(\ov{m}))-\lev(\phi(\ov{n}))\\
			\\
		 	& =\lev(\ov{m})-\lev(\ov{n})\\
		 	\\
		 	&=d_{\mathbb{T}_{\infty}}(\ov{n},\ov{m}).\\
			\end{split}
			\end{equation*}
		
		\item Case $\ov{n}$ and $\ov{m}$ are not comparable:
		
		Let $\ov{o}\in\mathbb{T}_{\infty}$ be the closest common ancestor to $\ov{n}$ and $\ov{m}$ in $\mtr$; and let $\mu\in\mks$ be the closest common ancestor to $\phi(\ov{n})$ and $\phi(\ov{m})$ in $\mks$.  Since $\phi(\ov{o})$ is also a common ancestor to $\phi(\ov{n})$ and $\phi(\ov{m})$, the definition of $\mu$ gives us that $\lev(\phi(\ov{o}))\leq \lev(\mu)$.  Hence %, using the fact that $\dist(\phi(\ov{n}),\phi(\ov{m}))$ is the length of a shortest path, 
		we have:
		
			\begin{equation*}
			\begin{split}
			d_{\mks}(\phi(\ov{n}),\phi(\ov{m})) &\leq d_{\mks}(\phi(\ov{n}),\mu)+d_{\mks}(\mu,\phi(\ov{m}))\\
			\\
			&= \lev(\phi(\ov{n}))-\lev(\mu)+\lev(\phi(\ov{m}))-\lev(\mu)\\
			\\
			&\leq \lev(\phi(\ov{n}))-\lev(\phi(\ov{o}))+\lev(\phi(\ov{m}))-\lev(\phi(\ov{o}))\\
			\\
			&=\lev(\ov{n})-\lev(\ov{o})+\lev(\ov{m})-\lev(\ov{o})\\
			\\
			&=d_{\mtr}(\ov{n},\ov{m}).\\
			\end{split}
			\end{equation*}
		\end{itemize}
	\item Let $\mu<\nu\in\mks$, and let $\ov{n}\in\phi^{-1}(\mu)$.  Choose {\sl a} path in $\mks$ from $\mu$ {\sl down} to $\nu$.  Record (every time such is defined) all the fraternity orders of all elements along that path.  Then, starting from $\mu$, lift that path node by node, inductively on the level of the node, as follows:
		\begin{itemize}
		\item $\phi(\ov{n})=\mu$.
		\item Assume all nodes on the path at distance less than or equal to $k$ from $\mu$ have been lifted.  Let $\zeta$ be on the path, at distance $k+1$ from $\mu$. Let $\zeta'$ be the immediate ancestor of $\zeta$ (which is well-defined since we are following a given path), and let $\ov{n'}\in \phi^{-1}(\zeta')$ be the lift of $\zeta'$ chosen according to the induction hypothesis.  We consider two cases.
			\begin{itemize}
			\item[$*$]If $\zeta'$ is a non-branching node, then lift $\zeta$ to any immediate descendant of $\ov{n'}$.
			\item[$*$] If $\zeta'$ is a branching node, then lift $\zeta$ to the immediate descendant of $\ov{n'}$ of the same fraternal order in $\mtr$ as $\zeta$ is among all immediate descendants of $\zeta'$.
			\end{itemize}
		\item Doing this way, we eventually get to an element $\ov{m}$ of $\mathbb{T}_{\infty}$ such that $\phi(\ov{m})=\nu$ with $\ov{m}>\ov{n}$ and $d_{\mks}(\mu,\nu)=d_{\mtr}(\ov{n},\ov{m})$.
		\end{itemize}
	\end{enumerate}
	\end{proof}
	
%Recall again the composition $\lambda=\phi\circ j\circ f_{|_{f^{-1}\left(j^{-1}(\mtr)\right)}}$.  Since $f$ is \LLD and \CLLD, the above properties of the map $\phi$ imply that the composition $\lambda$ is Lipschitz (for large distances) and \atd co-Lipschitz (for large distances).  We will apply the fork argument to the map $\lambda$.

\subsection{Fork argument}{\label{forksection}}
\mbox{}

The fork argument was introduced in \cite[proof of Theorem 4.1]{LimaRandrianarivony2010} for maps that are coarse co-Lipschitz.  Here we reduce its hypothesis when we apply it to a map $\lambda$ that is only \atd coarse co-Lipschitz.

First let us fix some notation.  Consider a metric space $\mathcal{S}$ and a surjective map $\lambda:\mS\tto\mks$ that is \atd coarse co-Lipschitz.  For ${\delta}>0$, consider the (nonempty) set  $\mathcal{C}({\delta})$ of all $c>0$ such that for all $R\geq {\delta}$, for all $\sigma\in\mS$, and for all $\nu>\lambda(\sigma)$ in $\mks$ with $d_{\mks}(\lambda(\sigma),\nu)<cR$, there exists $\sigma'\in\mS$ with $d_{{\mathcal{S}}}(\sigma,\sigma')\leq R$ and $\lambda(\sigma')=\nu$.  Denote by $c_{\delta}^{\ATD}$ the supremum of the set $\mathcal{C}({\delta})$.  It is left to the reader to check (see also \cite[Lemma 3.3]{LimaRandrianarivony2010}) that $c_{\delta}^{\ATD}$ is attained, and that the family $\left(c_{\delta}^{\ATD}\right)_{{\delta}>0}$ is increasing.  Denote by $\displaystyle c_{\infty}^{\ATD}:=\sup \left\{c_{\delta}^{\ATD}: {\delta}>0\right\}=\lim_{{\delta}\to\infty} c_{\delta}^{\ATD}$.
	
\medskip

With these notation, we present the fork argument for \atd coarse co-Lipschitz maps.
	
\begin{prop}\label{forkarg}
For a metric space $\mS$, let $\lambda:\mS\tto\mks$ be a surjective map that is \atd co-Lipschitz for large distances.  Assume that the quantity $c_{\infty}^{\ATD}$ for the map $\lambda$ is finite.  Then for every $\eps>0$, one can find $r>0$ as large as we want, as well as elements $\mu_0$, $\mu_1$, and $(\mu_{2,k})_{k\geq 1}$ in $\mks$, and elements $\sigma_0$, $\sigma_1$, and $(\sigma_{2,k})_{k\geq 1}$ in $\mS$ so that $\lambda(\sigma_0)=\mu_0$, $\lambda(\sigma_1)=\mu_1$, and $\lambda(\sigma_{2,k})=\mu_{2,k}$ for every $k\geq 1$; with the $\sigma$'s and the $\mu$'s sitting in an (approximate) fork position as follows:
$$d_{\mks}(\mu_0,\mu_1)=d_{\mks}(\mu_1,\mu_{2,k})=\frac{d_{\mks}(\mu_0,\mu_{2,k})}{2}=r;$$
and
\begin{equation*}
\begin{split}
d_{\mathcal{S}}(\sigma_0,\sigma_1) &\leq (1+3\eps)\,\ffrac{r}{c_{\infty}^{\ATD}},\\
d_{\mathcal{S}}(\sigma_1,\sigma_{2,k}) &\leq (1+3\eps)\,\ffrac{r}{c_{\infty}^{\ATD}},\\
\frac{d_{\mathcal{S}}(\sigma_0,\sigma_{2,k})}{2}&> (1-80\eps)\,\ffrac{r}{c_{\infty}^{\ATD}}.\\
\end{split}
\end{equation*}
\end{prop}
	\begin{proof}
	Let $\eps>0$.  Let ${\delta}_0=\displaystyle \frac{3^{n_0}}{c_{\infty}^{\ATD}}$ for some large enough $n_0$ such that $c_{\infty}^{\ATD}<(1+\eps)c_{{\delta}_0}^{\ATD}$.  Then for any ${\delta}\geq {\delta}_0$, we have
	$$c_{{\delta}_0}^{\ATD}\leq c_{\delta}^{\ATD}\leq c_{\infty}^{\ATD}<(1+\eps)c_{{\delta}_0}^{\ATD}\leq (1+\eps)c_{{\delta}}^{\ATD}\leq (1+\eps)c_{\infty}^{\ATD}.$$
	Since $c_{{\delta}}^{\ATD}$ is a maximum, since $(1+\eps)c_{{\delta}}^{\ATD}>c_{{\delta}}^{\ATD}$, and because of the ancestor-to-descendant coarse co-Lipschitz condition on $\lambda$, we can find a number $R\geq {\delta}$, and elements $\sigma\in\mS$ and $\nu>\lambda(\sigma)\in\mks$ with $c_{{\delta}}^{\ATD}R\leq d_{\mks}(\nu,\lambda(\sigma))<(1+\eps)c_{{\delta}}^{\ATD}R$ such that
	$$\forall \sigma'\in\mS, ~\left(\lambda(\sigma')=\nu\right)~\Longrightarrow~ \left(d_{\mathcal{S}}(\sigma,\sigma')> R\right).$$
	%On the other hand, the \atd coarse co-Lipschitz condition on $\lambda$ necessitates that $\dist(\nu,\lambda(\sigma))\geq c_{{\delta}}^{\ATD}R$.  
	In passing, note that we can then make our of choice of ${\delta}_0$ in such a way that $d_{\mks}(\nu,\lambda(\sigma))$ is as large as we want.
	
	\medskip
	
	Let us call $\mu:=\lambda(\sigma)$.  Let $N\in\mathbb{N}$ be such that $\displaystyle 3^N\leq \frac{d_{\mks}(\mu,\nu)}{2}<3^{N+1}$.  Since $\nu$ is a descendant of $\mu$, we have $d_{\mks}(\mu,\nu)=\lev(\nu)-\lev(\mu)$, and hence there must exist an integer $p\geq 0$ such that $\lev(\mu)\leq p3^N<(p+1)3^N\leq \lev(\nu)$.  In fact, one can take $p$ to be the largest integer such that $(p-1)3^N<\lev(\mu)$.  Furthermore, for ${\delta}\geq 81{\delta}_0$, we must necessarily have $3^{n_0}<3^{N-2}$.  In fact, ${\delta}\geq 54(1+\eps){\delta}_0$ if $\eps$ is small enough.  Then since $(1+\eps)\geq c_{\infty}^{\ATD}/c_{{\delta}}^{\ATD}$, and since $c_{\infty}^{\ATD}{\delta}_0=3^{n_0}$, we have
	$$2\cdot 3^{N+1}>d_{\mks}(\mu,\nu)\geq c_{{\delta}}^{\ATD}R\geq c_{{\delta}}^{\ATD}{\delta}\geq 6\cdot 9\cdot 3^{n_0}.$$
	
	\medskip
	
	We will now identify the elements $\mu_0$, $\mu_1$ and $\mu_{2,k}$($k\geq 1$) of $\mks$ that we need for our purpose.  First, we fix a path $\mathcal{P}$ downward from $\mu$ to $\nu$.  We consider the elements $a$ and $b$ of $\mathcal{P}$ at levels $p\cdot3^N$ and $(p+1)3^N$ respectively.  Note that the segment $[a,b]$ has length $3^N$.  Divide $[a,b]$ into three equal-length subsegments, and denote by $\mu_0$ the element of $\mathcal{P}$ at level $p\cdot3^N+3^{N-1}$, and by $\mu_3$ the one at level $p\cdot3^N+2\cdot3^{N-1}$.  Note that $d_{\mks}(\mu,\mu_0)\geq 3^{N-1}$ and $d_{\mks}(\mu_3,\nu)\geq 3^{N-1}$.  Divide the segment $[\mu_0,\mu_3]$ into three again, and consider the following vertices: the element $\mu_1$ of $\mathcal{P}$ at level $p\cdot3^N+3^{N-1}+3^{N-2}$, and the elements $\mu_{2,k}$ ($k\geq 1$) of $\mks$ at level $p\cdot 3^N+3^{N-1}+2\cdot3^{N-2}$ that lie between $\mu_1$ and the element $\mu_3$ already previously mentioned.  Note that although only one of the $\mu_{2,k}$'s belongs to the original path $\mathcal{P}$, the succession $\mu<\mu_0<\mu_1<\mu_{2,k}<\mu_3<\nu$ can be extended to a shortest path from $\mu$ to $\nu$ for any $k\geq 1$.
	
	\medskip
	
	Let us set $$d_{\mks}(\mu_1,\mu_0)=d_{\mks}(\mu_{2,k},\mu_1)=d_{\mks}(\mu_3,\mu_{2,k}):=r\,(=3^{N-2}\geq c_{\infty}^{\ATD}{\delta}_0).$$  Then note that $d_{\mks}(\mu_{2,k},\nu)=d_{\mks}(\mu,\nu)-d_{\mks}(\mu,\mu_0)-2r$ for each $k\geq 1$.
%	$$\begin{cases}
%	d_{\mks}(\mu,\mu_0):=D,\\
%	\\
%	d_{\mks}(\mu_1,\mu_0)=d_{\mks}(\mu_{2,k},\mu_1)=d_{\mks}(\mu_3,\mu_{2,k}):=r\,(=3^{N-2}\geq c_{\infty}^{\ATD}{\delta}_0),\\
%	\\
%	d_{\mks}(\mu_3,\nu)=d_{\mks}(\mu,\nu)-D-3r.\\
%	\end{cases}$$
	We will now lift these elements one by one starting from $\mu$ using the \atd coarse co-Lipschitz condition.  
	\begin{enumerate}
	\item $d_{\mks}(\mu,\mu_0)<c_{{\delta}_0}^{\ATD}\cdot\ffrac{(1+\eps)}{c_{{\delta}_0}^{\ATD}}\,d_{\mks}(\mu,\mu_0)$; and $\ffrac{(1+\eps)}{c_{{\delta}_0}^{\ATD}}\,d_{\mks}(\mu,\mu_0)\geq {\delta}_0$ because $d_{\mks}(\mu,\mu_0)\geq 3^{N-1}\geq c_{\infty}^{\ATD}{\delta}_0\geq c_{{\delta}_0}^{\ATD}{\delta}_0$.  So there exists an element $\sigma_0\in\mS$ with $d_{\mathcal{S}}(\sigma,\sigma_0)\leq \ffrac{1+\eps}{c_{{\delta}_0}^{\ATD}}\,d_{\mks}(\mu,\mu_0)$ such that $\lambda(\sigma_0)=\mu_0$.
	
	\item $d_{\mks}(\mu_0,\mu_1)=r<c_{{\delta}_0}^{\ATD}\cdot\ffrac{1+\eps}{c_{{\delta}_0}^{\ATD}}\,r$; and $\ffrac{1+\eps}{c_{{\delta}_0}^{\ATD}}\,r\geq {\delta}_0$ since $r=3^{N-2}\geq c_{\infty}^{\ATD}{\delta}_0$.  So there is an element $\sigma_1\in\mS$ with $d_{\mathcal{S}}(\sigma_0,\sigma_1)\leq \ffrac{1+\eps}{c_{{\delta}_0}^{\ATD}}\,r$ such that $\lambda(\sigma_1)=\mu_1$.
	
	\item Similarly, $d_{\mks}(\mu_1,\mu_{2,k})=r$ for each $k\geq 1$, so there is an element $\sigma_{2,k}\in\mS$ with $d_{\mathcal{S}}(\sigma_1,\sigma_{2,k})\leq \ffrac{1+\eps}{c_{{\delta}_0}^{\ATD}}\,r$ such that $\lambda(\sigma_{2,k})=\mu_{2,k}$.
	
	\item Finally, $d_{\mks}(\mu_{2,k},\nu)=d_{\mks}(\mu,\nu)-d_{\mks}(\mu,\mu_0)-2r<c_{{\delta}_0}^{\ATD}\cdot\ffrac{1+\eps}{c_{{\delta}_0}^{\ATD}}\,[d_{\mks}(\mu,\nu)-d_{\mks}(\mu,\mu_0)-2r]$; and $\ffrac{1+\eps}{c_{{\delta}_0}^{\ATD}}\,[d_{\mks}(\mu,\nu)-d_{\mks}(\mu,\mu_0)-2r]\geq\ffrac{1+\eps}{c_{{\delta}_0}^{\ATD}}\,r\geq {\delta}_0$, so there exists an element $\sigma^{(k)}\in\mS$ with $d_{\mathcal{S}}(\sigma_{2,k},\sigma^{(k)})\leq \ffrac{1+\eps}{c_{{\delta}_0}^{\ATD}}\,[d_{\mks}(\mu,\nu)-d_{\mks}(\mu,\mu_0)-2r]$ such that $\lambda(\sigma^{(k)})=\nu$.
	\end{enumerate}
	
	\medskip
	
%	Now, by the choice of $\mu$ and $\nu$, we have $\dist(\sigma,\sigma^{(k)})>R$.  So
%	\[
%	\begin{split}
%	R<\dist(\sigma,\sigma^{(k)})&\leq \ffrac{1+\eps}{c_{{\delta}_0}^{\ATD}}\,[D+r+r+(\dist(\mu,\nu)-D-2r)]\\
%	\\
%	&=\ffrac{1+\eps}{c_{{\delta}_0}^{\ATD}}\,\dist(\mu,\nu)\\
%	\\
%	&<\ffrac{1+\eps}{c_{{\delta}_0}^{\ATD}}\,(1+\eps)c_{\infty}^{\ATD}R\\
%	\\
%	&<(1+\eps)^3R.\\
%	\end{split}
%	\]

	We claim that as a result, we must have 
	$$d_{\mathcal{S}}(\sigma_{2,k},\sigma_0)>[1-(3^4-1)\eps]\cdot\ffrac{2r}{c_{{\delta}_0}^{\ATD}}.$$
	In fact by the choice of $\mu$ and $\nu$, we have $d_{\mathcal{S}}(\sigma,\sigma^{(k)})>R$.  So  if $d_{\mathcal{S}}(\sigma_{2,k},\sigma_0)\leq\alpha\cdot\ffrac{2r}{c_{{\delta}_0}^{\ATD}}$, then
	\begin{equation*}
	\begin{split}
	R&<d_{\mathcal{S}}(\sigma,\sigma^{(k)})\\
	\\
	&\leq d_{\mathcal{S}}(\sigma,\sigma_0)+d_{\mathcal{S}}(\sigma_0,\sigma_{2,k})+d_{\mathcal{S}}(\sigma_{2,k},\sigma^{(k)})\\
	\\
	&\leq \frac{1+\eps}{c_{{\delta}_0}^{\ATD}}\,d_{\mks}(\mu,\mu_0)+\frac{\alpha}{c_{{\delta}_0}^{\ATD}}\,2r+\frac{1+\eps}{c_{{\delta}_0}^{\ATD}}\,(d_{\mks}(\mu,\nu)-d_{\mks}(\mu,\mu_0)-2r)\\
	\\
	&=\frac{\alpha}{c_{{\delta}_0}^{\ATD}}\,2r+\frac{1+\eps}{c_{{\delta}_0}^{\ATD}}\,(d_{\mks}(\mu,\nu)-2r)\\
	\\
	&=\frac{d_{\mks}(\mu,\nu)}{c_{{\delta}_0}^{\ATD}}\,\left[t\alpha+(1-t)(1+\eps)\right] \text{, where } t=\frac{2r}{d_{\mks}(\mu,\nu)}\\
	\\
	&<\frac{(1+\eps)c_{\infty}^{\ATD}R}{c_{{\delta}_0}^{\ATD}}\,[t\alpha+(1-t)(1+\eps)]\\
	\\
	&<(1+\eps)^2[t\alpha+(1-t)(1+\eps)]R.
	\end{split}
	\end{equation*}
	This would require that $\alpha>(1+\eps)-\ffrac{1}{t}\left[(1+\eps)-\frac{1}{(1+\eps)^2}\right]$.  But from the definition of $N$, we have $d_{\mks}(\mu,\nu)<2\cdot3^{N+1}$.  So, since $r=3^{N-2}$, we have
	$$\frac{1}{t}=\frac{d_{\mks}(\mu,\nu)}{2r}<3^3,$$
	and hence
	$$
	\alpha>(1+\eps)-3^3\left[(1+\eps)-\frac{1}{(1+\eps)^2}\right]\geq(1+\eps)-3^3(3\eps).
	$$

	\medskip
	
	To summarize, we have $\sigma_0\in\lambda^{-1}(\mu_0)$, $\sigma_1\in\lambda^{-1}(\mu_1)$, and $\sigma_{2,k}\in\lambda^{-1}(\mu_{2,k})$ ($k\geq 1$), and with $\eps$ small enough we have
	$$d_{\mathcal{S}}(\sigma_0,\sigma_1)\leq (1+\eps)\,\frac{r}{c_{\delta_0}^{\ATD}}<(1+\eps)^2\,\frac{r}{c_{\infty}^{\ATD}}\leq(1+3\eps)\,\frac{r}{c_{\infty}^{\ATD}},$$
	and
	$$d_{\mathcal{S}}(\sigma_1,\sigma_{2,k})\leq (1+\eps)\,\frac{r}{c_{\delta_0}^{\ATD}}<(1+\eps)^2\,\frac{r}{c_{\infty}^{\ATD}}\leq(1+3\eps)\,\frac{r}{c_{\infty}^{\ATD}},$$
	but
	$$d_{\mathcal{S}}(\sigma_0,\sigma_{2,k})> (1-80\eps)\,\frac{2r}{c_{{\delta}_0}^{\ATD}}\geq (1-80\eps)\,\frac{2r}{c_{\infty}^{\ATD}}.$$
	This finishes the proof.
	\end{proof}

\subsection{Proof of Proposition \ref{mainprop}}
\mbox{}

We apply Proposition \ref{forkarg} to the set $\mathcal{S}:=f^{-1}\left(\Sigma\right)=f^{-1}\left(j^{-1}(\mtr)\right)$ and the map $\lambda=\phi\circ j\circ f_{|_{\mS}}$.  Recall that this particular map is also Lipschitz, and hence the hypothesis of Proposition \ref{forkarg} applies since $c_{{\delta}}^{\ATD}\leq \Lip(\lambda)$ for all ${\delta}>0$.  We find the elements $\mu_0, \mu_1, (\mu_{2,k})_ {k\geq 1}\in\mks$, and $\sigma_0, \sigma_1, (\sigma_{2,k})_{k\geq 1}\in\mS\subseteq X$ as given by Proposition \ref{forkarg}.  We have for all $k\geq 1$,
	$$\|\sigma_0-\sigma_1\|_{_X}\leq (1+3\eps)\,\frac{r}{c_{\infty}^{\ATD}},$$
	and
	$$\|\sigma_1-\sigma_{2,k}\|_{_X}\leq (1+3\eps)\,\frac{r}{c_{\infty}^{\ATD}},$$
	but
	$$\|\sigma_0-\sigma_{2,k}\|_{_X}> (1-80\eps)\,\frac{2r}{c_{\infty}^{\ATD}}.$$
	
	\medskip
	
On the other hand, since $\lambda$ is Lipschitz , we have for all $k\neq l$,
	$$\|\sigma_{2,k}-\sigma_{2,l}\|_{_X}\geq \frac{1}{\Lip(\lambda)}\,d_{\mks}(\mu_{2,k},\mu_{2,l})=\frac{1}{\Lip(\lambda)}\,2r.$$
Since $X$ has property $(\beta)$, this implies that we must have for $\eps$ small enough,
$$\ov{\beta}_X\left(\frac{c_{\infty}^{\ATD}}{\Lip(\lambda)}\right)\leq\ov{\beta}_X\left(\frac{2c_{\infty}^{\ATD}}{(1+3\eps)\Lip(\lambda)}\right)\leq \frac{83\eps}{1+3\eps}.$$
Since $\eps$ is arbitrary, we then get that
$$\ov{\beta}_X\left(\frac{c_{\infty}^{\ATD}}{\Lip(\lambda)}\right)=0.$$
This contradiction finishes the proof.

\hfill{$\square$}
%\begin{rmk}
%Introduce $\mathbb{L}_{\infty}$.  Mention that the argument also works with $\mathbb{L}_{\infty}$.  Then maybe note that the 2-branching version has bounded degree (so is doubling?)
%\end{rmk}

\medskip

\section{Removing the reflexivity assumption}{\label{second note}}

James proves many characterizations of reflexivity in \cite{James1964a}.  We highlight in the following lemma the one that we need.  For this purpose, denote by $\btr$ the set of all finite subsets of $\mathbb{N}$, considered in graph-theoretic terms as a tree.  Note that this time $\btr$ is only considered as a tree to organize its elements, but we do not put any metric on $\btr$.  As a set, we have $\btr=\mtr$.

\begin{lem}[James]{\label{James31}}
Let $X$ be a non-reflexive Banach space.  Then for every $0<\theta<1$, there exists a sequence $(u_n)_{n\geq 1} \subseteq B_X$ such that
	\begin{enumerate}
	\item the map from $\btr$ to $X$ given by $\{n_1<n_2<\cdots <n_k\}\longmapsto u_{n_1}+\cdots+u_{n_k}$ is one-to-one,
	\item for every $n_1<\cdots<n_k$, we have
	$$\theta k \leq \|u_{n_1}+\cdots+u_{n_k}\|\leq k,$$
	\item and for every $n_1<\cdots<n_k<m_1<\cdots<m_l$, we have
	$$\frac{\theta}{3}(k+l)\leq \|(u_{n_1}+\cdots+u_{n_k})-(u_{m_1}+\cdots+u_{m_l})\|\leq k+l.$$
	\end{enumerate}
\end{lem}
	\begin{proof}
	Let $\theta\in(0,1)$.  From item (31) in \cite{James1964a}, we can find sequences $(u_n)_{n\geq 1}\subseteq B_X$ and $(u_n^*)_{n\geq 1} \subseteq B_{X^*}$ such that $u_n^*(u_k)=\theta$ if $n\leq k$, and $u_n^*(u_k)=0$ if $n>k$.\\
		\begin{enumerate}
		\item Let $\{n_1<\cdots<n_k\}\neq \{m_1<\cdots<m_l\}$.  Denote $\{n'_1<\cdots<n'_{s}\}=\{n_1<\cdots<n_k\}\Delta \{m_1<\cdots<m_l\}$, where $s\leq k+l$; and let signs $(\eps_i)_{1\leq i\leq s}\subseteq \{-1,1\}$ be such that
		$$(u_{n_1}+\cdots+u_{n_k})-(u_{m_1}+\cdots+u_{m_l})=\eps_1u_{n'_1}+\cdots+\eps_{s}u_{n'_{s}}.$$
		Then we have
		$$\|(u_{n_1}+\cdots+u_{n_k})-(u_{m_1}+\cdots+u_{m_l})\|\geq \left|u_{n'_{s}}\left(\eps_1u_{n'_1}+\cdots+\eps_{s}u_{n'_{s}}\right)\right|=|\eps_{s}\,\theta|>0.$$\\
		
		\item For $n_1<\cdots<n_k$, we have:
		$$\|u_{n_1}+\cdots+u_{n_k}\|\geq u_{n_1}^*(u_{n_1}+\cdots+u_{n_k})=\theta\,k.$$
		The other inequality follows from the triangle inequality.\\
		
		\item Let $n_1<\cdots<n_k<m_1<\cdots<m_l$.  On the one hand, we have
		$$\|(u_{n_1}+\cdots+u_{n_k})-(u_{m_1}+\cdots+u_{m_l})\|\geq \left|u_{m_1}^*\left((u_{n_1}+\cdots+u_{n_k})-(u_{m_1}+\cdots+u_{m_l})\right)\right|=\theta\,l,$$
		and on the other hand,
		$$\|(u_{n_1}+\cdots+u_{n_k})-(u_{m_1}+\cdots+u_{m_l})\|\geq \left|u_{n_1}^*\left((u_{n_1}+\cdots+u_{n_k})-(u_{m_1}+\cdots+u_{m_l})\right)\right|=\theta|k-l|.$$
		We get the left inequality by noting that $\max\{l,|k-l|\}\geq \frac13 (k+l)$.		
		Again the inequality on the right follows from the triangle inequality.
		\end{enumerate}
	\end{proof}
	
\medskip

We now state the main result in this section.

\begin{thm}{\label{thm2}}
Assume that a Banach space $Y$ is a uniform quotient of a Banach space $X$.  If $X$ has (an equivalent norm with) property $(\beta)$, then $Y$ has to be reflexive.
\end{thm}
	\begin{proof}
	Let $f:X\tto Y$ be a surjective uniform quotient map.  Note that $f$ is \LLD and coarse co-Lipschitz.  Assume that $Y$ is not reflexive.  We apply James' characterization with $\theta=3/4$.  For $J=\{n_1<n_2<\ldots<n_k\}\in\btr$, set $v_{J}:=u_{n_1}+\ldots+u_{n_k}$.  We then get that if $\max J<\min J'$, then
	\begin{equation}{\label{James}}
	\begin{cases}
	\ffrac14 |J|\leq\|v_J\|_{_Y}\leq |J|,\\
	\\
	\ffrac14 |J'|\leq\|v_{J'}\|_{_Y}\leq |J'|,\\
	\\
	\ffrac14 \left(|J|+|J'|\right)\leq\|v_J-v_{J'}\|_{_Y}\leq \left(|J|+|J'|\right).
	\end{cases}
	\end{equation}
	Consider the subset $\ctr:=\{v_{J} \in Y: J\in\btr\}\subseteq Y$, which {\sl is} a metric space.  Since the correspondence $J\longmapsto v_{J}$ is one-to-one, the set $\ctr$ is in one-to-one and onto correspondence with $\btr$, and we extend the notions of levels, ancestors, and descendants to elements of $\ctr$ as well.  Consider the composition
	$$\ttilde{\phi}:=(\phi\circ {_{^\equiv}}): \ctr ~\equiv~ \mtr \tto\mks,$$
	where $\phi$ is the same map as in Subsection \ref{phi}, and $\equiv$ is the natural bijection between $\ctr$ and $\mtr$.  The bijection $\ctr \equiv \mtr$ is \atd bi-Lipschitz by James' characterization as expressed in (\ref{James}) above.  And hence the composition $\ttilde{\phi}: \ctr  \tto \mks$, considered as a map between metric spaces, is \atd co-Lipschitz.

	\medskip

	Next, consider the restriction $f_{|_{{f^{-1}(\ctr)}}}:f^{-1}(\ctr)\tto \ctr\subseteq Y$, which is a surjective uniform quotient mapping that is \LLD and coarse co-Lipschitz.  As a result, the composition $\displaystyle \ttilde{\lambda}:=\ttilde{\phi} \circ f_{|_{{f^{-1}(\ctr)}}}: f^{-1}(\ctr)\tto \mks$ is also \atd co-Lipschitz.  Note that $c_{\infty}^{\ATD}(\ttilde{\lambda})$ is finite since the inequalities (\ref{James}) also give that the map $\ttilde{\phi}$ is \atd Lipschitz.  We can thus apply the fork argument (Proposition \ref{forkarg}) to $\ttilde{\lambda}$, producing elements $\mu_0$, $\mu_1$, $\mu_{2,k}$ ($k\geq 1$) in $\mks$ with $d_{\mks}(\mu_0,\mu_1)=d_{\mks}(\mu_1,\mu_{2,k})=d_{\mks}(\mu_0,\mu_{2,k})/2=r=3^{N-2}$; as well as respective $\ttilde{\lambda}$-preimages $\sigma_0$, $\sigma_1$, $\sigma_{2,k}$ ($k\geq 1$) in $f^{-1}(\ctr)\subseteq X$ with
	\begin{equation}{\label{smallseparation}}
	\begin{cases}
	\|\sigma_0-\sigma_1\|_{_X},\|\sigma_1-\sigma_{2,k}\|_{_X}\leq (1+3\eps)\,\ffrac{r}{c_{\infty}^{\ATD}(\ttilde{\lambda})}\\
	\\
	\|\sigma_0-\sigma_{2,k}\|_{_X}> (1-80\eps)\,\ffrac{2r}{c_{\infty}^{\ATD}(\ttilde{\lambda})}.\\
	\end{cases}
	\end{equation}
	%By using property $(\beta)$ in $X$, we get that $\sep\left((\sigma_{2,k})_{k\geq 1}\right)\leq small\cdot r$.

	\medskip

	Let us denote $v^{(0)}:=f(\sigma_0)$, $v^{(1)}:=f(\sigma_1)$, and $v_{{J^{(k)}}}:=f(\sigma_{2,k})$ (for all $k\geq 1$), which  are elements of the subset $\ctr$ of $Y$.  Recall that the map $\phi$ preserves levels, so
	$$\begin{cases}
	\lev(v^{(0)})=\lev(\mu_0),\\
	\\
	\lev(v^{(1)})=\lev(\mu_1)=\lev(\mu_0)+3^{N-2},\\
	\\
	\lev(v_{{J^{(k)}}})=\lev(\mu_{2,k})=\lev(\mu_1)+3^{N-2}.\\
	\end{cases}
	$$
	But $\lev(v_{{J^{(k)}}})=|J^{(k)}|$, so $|J^{(k)}|=\lev(\mu_1)+3^{N-2}$ for all $k\geq 1$.

	\medskip

	Consider the set of all elements of $\btr$ with infinitely (not necessarily immediate) descendants among the $J^{(k)}$'s.  This set is nonempty since the root $\emptyset$ belongs to it.  Now let $J$ be an element of this set such that $\lev(J)=|J|$ is maximal.  Let $\left(J^{(k_n)}\right)_{n\geq 1}$ be the infinitely many descendants of $J$ among the $J^{(k)}$'s.  Write $J^{(k_n)}=:J\cup I^{(n)}$ with $\max J<\min I^{(n)}$.  Then for every $n\geq 1$, we must have
	\begin{equation}{\label{cardinality}}
	|I^{(n)}|\geq 3^{N-2}.
	\end{equation}
	In fact, the map $\phi: \mtr\tto \mks$ is $1$-Lipschitz, so for any $k\neq j$ we have
	$$2\cdot 3^{N-2} =d_{\mks}(\mu_{2,k},\mu_{2,j})\leq d_{\mtr}(J^{(k)},J^{(j)}),$$
	i.e. the closest common ancestor to $J^{(k)}$ and $J^{(j)}$ is at least $3^{N-2}$ generations back.
	
	\medskip

	Next, consider the immediate descendants $\left(J\cup \{m\}\right)_{m>\max(J)}$ of $J$.  By the fact that $\lev(J)$ is maximal among the elements of $\btr$ with infinitely many descendants among the $J^{(k)}$'s, we must find an infinite sequence $m_1<m_2<\cdots$ such that 
		\begin{itemize}
		\item each $J\cup\{m_j\}$ has descendants among the $J^{(k_n)}$'s,
		\item and each such $J\cup\{m_j\}$ has only finitely many descendants among the $J^{(k_n)}$'s.
		\end{itemize}
	Take a further subsequence of $(m_j)_{j\geq 1}$, still denoted $(m_j)_{j\geq 1}$, such that if we denote by $J^{\left(k_{n_j}\right)}=J\cup I^{(n_j)}$ the %({\sc PICK ONE})
	descendant of $J\cup\{m_j\}$ with $\max\left(I^{(n_j)}\right)$ being maximal, then $\max(I^{(n_j)})<m_{j+1}=\min(I^{(n_{j+1})})$.
	
	Then, by James, we have
	$$\left\|v_{{J^{(k_{n_j})}}}-v_{{J^{(k_{n_{j'}})}}}\right\|_{_Y}\geq\ffrac14\left(|I^{(n_j)}|+|I^{(n_{j'})}|\right)\geq \frac12\cdot3^{N-2}.$$
	(The last inequality comes from inequality (\ref{cardinality}).)
	Using the Lipschitz condition for large distances on $f$, we get
	$$\left\|\sigma_{2,k_{n_j}}-\sigma_{2,k_{n_{j'}}}\right\|_{_X}\geq \ffrac1L\cdot 3^{N-2}=\ffrac{1}{L} \cdot r$$
	for some constant $L>0$ independent of $\eps$.  Indeed we can take $L$ to be such that $L/2$ is the Lipschitz constant for $f$ for distances larger than $\inf\{\|\sigma_{2,k_{n_j}}-\sigma_{2,k_{n_{j'}}}\|_{_X}:j\neq j'\}$.  Note that since $f$ is uniformly continuous, $\inf\{\|\sigma_{2,k_{n_j}}-\sigma_{2,k_{n_{j'}}}\|_{_X}:j\neq j'\}$ is strictly positive and independent of $\eps$.  In fact we have $\inf\{\|v_{J^{(k)}}-v_{J^{(k')}}\|_{_Y}:k\neq k'\}\geq \theta=3/4$ by the proof of item (1) in Lemma \ref{James31}.  So $\inf\{\|\sigma_{2,k_{n_j}}-\sigma_{2,k_{n_{j'}}}\|_{_X}:j\neq j'\}\geq \eta$, where $\eta>0$ is such that $\|\sigma-\sigma'\|_{_X}<\eta$ implies $\|f(\sigma)-f(\sigma')\|_{_Y}<3/4$.
	
	\medskip
	
	Combining this with the distances expressed in (\ref{smallseparation}), and using property $(\beta)$ on $X$, we must have for $\eps$ small enough,
$$\ov{\beta}_X\left(\frac{c_{\infty}^{\ATD}(\ttilde{\lambda})}{2L}\right)\leq\ov{\beta}_X\left(\frac{c_{\infty}^{\ATD}(\ttilde{\lambda})}{(1+3\eps)L}\right)\leq \frac{83\eps}{1+3\eps}.$$
Since $\eps$ is arbitrary, we then get that
$$\ov{\beta}_X\left(\frac{c_{\infty}^{\ATD}(\ttilde{\lambda})}{2L}\right)=0,$$
a contradiction.
	\end{proof}

\begin{rmk}
A superreflexive Banach space is one that has an equivalent norm which is uniformly convex (see for example \cite[Appendix A]{BenyaminiLindenstrauss2000}).  A result of Bates, Johnson, Lindenstrauss, Preiss, and Schechtman \cite{BatesJohnsonLindenstraussPreissSchechtman1999} shows that a Banach space that is a uniform quotient of a superreflexive Banach space is itself superreflexive.  Our result (by combining Theorem \ref{thm1} and Theorem \ref{thm2}) gives a generalization of this for the case of separable Banach spaces.  In fact, a uniformly convex Banach space has property $(\beta)$ (see \cite{Rolewicz1986}); but there are separable spaces with property $(\beta)$ which are not superreflexive (see for example \cite{Kutzarova1989}).

In \cite{BaudierKaltonLancien2010} Baudier, Kalton and Lancien show that a Banach space which coarse Lipschitz embeds into a separable Banach space with an equivalent norm with property $(\beta)$ also has an equivalent norm with property $(\beta)$.  Our result here gives the analog of this for uniform quotients.
\end{rmk}

\medskip
	
\section{Quantitative results}{\label{Denka's notes}}

Recall the following asymptotic moduli.

\begin{definition}  Let $X$ be an infinite-dimensional Banach space.  Let $t\in (0,1]$.
\begin{enumerate}

\item Modulus of asymptotic uniform convexity (AUC) \cite{Milman1971, Huff1980, JohnsonLindenstraussPreissSchechtman2002}:
\[
\ov{\delta}(t) = \inf_{\|x\|=1} ~\sup_{X_0} ~\inf_{{z\in X_0}\atop{{~}\atop{\|z\|\geq t}}} \{ \|x+z\|-1\},
\]
where $X_0$ runs through all closed subspaces of $X$ of finite codimension.  The Banach space $X$ is said to be asymptotically uniformly convex (AUC) if $\ov{\delta}(t)>0$ for $t>0$.  A reflexive Banach space that is AUC is in particular called nearly uniformly convex (NUC).

\item Modulus of asymptotic uniform smoothness (AUS) \cite{Milman1971, Prus1989, JohnsonLindenstraussPreissSchechtman2002}:
\[
\ov{\rho}(t) = \sup_{\|x\|=1} ~\inf_{X_0} ~\sup_{{z\in X_0}\atop{{~}\atop{\|z\|\leq t}}} \{ \|x+z\|-1\},
\]
where $X_0$ runs through all closed subspaces of $X$ of finite codimension.  The Banach space $X$ is said to be asymptotically uniformly smooth (AUS) if $
\displaystyle \lim_{t\to 0}\ov{\rho}(t)/t=0$.  A reflexive Banach space that is AUS is in particular called nearly uniformly smooth (NUS).

\item $(\beta)$-modulus \cite{AyerbeDominguez_BenavidesCutillas1994}:
\[
\ov{\beta}(t) = 1 - \sup \left\{ \inf \left\{
\frac{\|x+x_n\|}{2} : n\geq 1 \right\} \right\},
\]
where the supremum is taken over all $x\in B_X$ and all sequences $(x_n)_{n\geq 1}\subseteq B_X$ with $\sep\left((x_n)_{n\geq 1}\right)\geq t$.
The Banach space $X$ is said to have property $(\beta)$ if $\ov{\beta}(t)>0$ for $t>0$.
\end{enumerate}
\end{definition}

\medskip

\begin{lem}{\label{moduli}}
For any infinite-dimensional Banach space $X$, and any $t\in (0,1/2]$, we have $\ov{\beta}_X(t)\leq \ov{\delta}_X(2t)$.
\end{lem}

\begin{proof}
For $t\in(0,1/2]$, denote $\beta=\ov{\beta}_X(t)$ and assume without loss of generality that $\beta>0$.  Let $x\in S_X$.  Then for every sequence $(u_n)_{n\geq 1} \subseteq B_X$ with $\sep\left((u_n)_{n\geq 1}\right)\geq t$, one has
$$\inf\left\{ \frac{\|x+u_n\|}{2}:n\geq 1\right\} \leq 1-\beta<\frac{1}{1+\beta}.$$
Multiplying everything by $1+\beta$ and taking the contrapositive, we get for any sequence $(x_n)_{n\geq 1}$ that if $\sep\left((x_n)_{n\geq 1}\right)\geq (1+\beta)\,t$, and if $\displaystyle \inf_{n\geq 1}\|(1+\beta)x+x_n\|\geq 2$, then $\|x_n\|>1+\beta$ on a subsequence, and hence $\liminf \|x_n\|\geq 1+\beta$.

\medskip

Denote $T=2t$.  For a finite codimensional subspace $X_0$, denote
$$\ov{\delta}_X(T,x,X_0)=\inf\{\|x+z\|-1: z\in X_0, \|z\|\geq T\};$$
and denote $\ov{\delta}_X(T,x)$ the supremum of $\ov{\delta}_X(T,x,X_0)$ over all finite codimensional subspaces $X_0$.

Fix $\eps>0$.  Let $x^*\in S_{X^*}$ be such that $x^*(x)=\|x\|=1$.  Denote by $X_1=\ker(x^*)$.  Let $z_1\in X_1$ be such that $\|z_1\|\geq T$ and
$$\|x+z_1\|-1<\ov{\delta}_X(T,x,X_1)+\eps\leq \ov{\delta}_X(T,x)+\eps.$$
Let $z_1^*\in S_{X^*}$ be such that $z_1^*(z_1)=\|z_1\|$.  Call $X_2=X_1\cap \ker(z_1^*)$, and let $z_2\in X_2$ be such that $\|z_2\|\geq T$ and 
$$\|x+z_2\|-1<\ov{\delta}_X(T,x,X_2)+\eps\leq \ov{\delta}_X(T,x)+\eps.$$
Continuing by induction, we find a sequence $(z_n)_{n\geq 1}\subseteq \ker(x^*)$ and functionals $(z_n^*)_{n \geq 1}\subseteq S_{X^*}$ such that for every $n\geq 1$ one has $z_n^*(z_n)=\|z_n\|\geq T$, and $z_{n+1}\in \bigcap_{i=1}^{n}\ker(z_i^*)$, and
$$\|x+z_n\|-1\leq \ov{\delta}_X(T,x)+\eps.$$

\medskip

Denote $x_n=x+z_n$.  Then for $m>n$,
$$\|x_n-x_m\|=\|z_n-z_m\|\geq z_n^*(z_n-z_m)=z_n^*(z_n)=\|z_n\|\geq T.$$
So $\sep\left((x_n)_{n\geq 1}\right)\geq T=2t\geq (1+\beta)\,t$.  (Note that $\beta\leq 1$ by definition.)  On the other hand, we have
$$\|(1+\beta)x+x_n\|=\|(2+\beta)x+z_n\|\geq x^*((2+\beta)x+z_n)=2+\beta> 2.$$
As a result, we must have $\liminf \|x_n\|=\liminf \|x+z_n\|\geq 1+\beta$.

\medskip

This implies that $\beta=\ov{\beta}_X(t)\leq \ov{\delta}_X(2t,x)+\eps$, and since $\eps>0$ and $x\in S_X$ are arbitrary, we get $\ov{\beta}_X(t)\leq \ov{\delta}_X(2t)$.

\end{proof}

\medskip

\begin{definition} We say that the $(\beta)$-modulus has power type $p$ if there is a constant $C>0$ so that $\ov{\beta}(t) \ge Ct^p$.
\end{definition}

\medskip

\begin{prop}
Let $1<p<\infty$. There exist uniformly homeomorphic Banach spaces $C_p$
and $K_p$ so that $C_p$ has $(\beta)$-modulus of power type $p$ while $K_p$ has no
equivalent norm with $(\beta)$-modulus of power type $p$.
\end{prop}
\begin{proof}
Let $(G_n)_{n\geq 1}$ be a sequence of finite-dimensional Banach spaces dense
in all finite-dimensional Banach spaces for the Banach-Mazur
distance. Let $1<p<\infty$. Johnson \cite{Johnson1971} introduced the space $C_p$ given by $C_p = ( \sum_{n=1}^{\infty} G_n)_{\ell_p}$.
Kalton \cite{Kalton2010} considered the space $K_p = \left(\left( \sum_{n=1}^{\infty} G_n\right)_{T_p}\right)^2$, where $T_p$ is the $p$-convexification of Tsirelson space $T$ (see for example \cite{CasazzaShura1989}).  Kalton proved in \cite{Kalton2010} that $C_p$ and $K_p$ are uniformly homeomorphic. 

The computation in \cite[Proposition 5.1]{DilworthKutzarovaLancienRandrianarivony2012} gives that $C_p$ has 
$(\beta)$-modulus of power type $p$. On the other hand, Kalton remarks in \cite{Kalton2010} that $K_p$ does not admit any equivalent norm with AUC modulus of power type $p$.  Assume that $K_p$ admits an equivalent norm with $(\beta)$-modulus of power type
$p$, that is $\ov{\beta}_{K_p}(t) \ge Ct^p$. By Lemma \ref{moduli}, we have $\ov{\beta}_{K_p}(t) \le \ov{\delta}_{K_p}(2t)$. Thus for this norm we would obtain $\ov{\delta}_{K_p}(t) \ge C't^p$, which is in contradiction with Kalton's result.
\end{proof}

%\medskip

\begin{rmk} In contrast with Theorems \ref{thm1} and \ref{thm2}, as well as with \cite[Proposition 4.1]{MendelNaor2010}, the above proposition shows that we cannot compare $\ov{\beta}_{C_p}(t)$ and $\ov{\beta}_{K_p}(t)$ even under renorming, despite the fact that one is a uniform quotient of the other.

Furthermore, it was proved in \cite{DilworthKutzarovaLancienRandrianarivony2012} that if a Banach space $Y$ is a uniform quotient of a Banach space $Z$, then for some constants  $C_1$ and $C_2$ we have,
\[
C_1\ov{\beta}_Z(C_2t) \le \ov{\rho}_Y(t).
\]
Taking $Z=C_p$ and $Y=K_p$, it follows from the above proposition that one can not replace the AUS modulus $\ov{\rho}_Y$ with the AUC modulus $\ov{\delta}_Y$ even under an equivalent norm on $Y$.
\end{rmk}

\medskip

The rest of the article is devoted to an investigation of a renorming question.  Recall that the AUS modulus is of power type $p$ if there exists a constant $C$ such that $\ov{\rho}(t)\leq Ct^p$.  In \cite{GodefroyKaltonLancien2001} it is shown that if a separable Banach space $X$ with AUS modulus of power type $p>1$ is uniformly homeomorphic to another Banach space $Y$, then for any $q<p$, $Y$ can be renormed to have AUS modulus of power type $q$.  It is natural to ask whether the same holds for the $(\beta)$-modulus, namely:

\medskip

\begin{quote}
Assume that two Banach spaces are uniformly homeomorphic and that one of them has $(\beta)$-modulus of power type $p$. Given an arbitrary $q>p$, does the other space always have an equivalent norm with $(\beta)$-modulus of power type $q$?
\end{quote}

\medskip

In the case of the counterexample provided by the pair of spaces $C_p$ and $K_p$, where $1<p<\infty$, we can give a positive answer to the above question.

\medskip

\begin{prop}
Assume that $X$ is a separable reflexive Banach space which is uniformly homeomorphic to a space with $\ov{\rho}(t) \le C_1t^p$, $1<p<\infty$, and that the dual space $X^*$ is uniformly homeomorphic to a space with  $\ov{\rho}(t) \le C_2t^{p'}$ where $1/p + 1/p' =1$. Then for any $\gamma >0$, $X$ admits an equivalent norm with $\ov{\beta}(t) \ge Ct^{p+\gamma}$.
\end{prop}
\begin{proof}
By the result in \cite{GodefroyKaltonLancien2001}, the space $X$ has an equivalent norm $\|.\|_1$ with $\ov{\rho}_{\|.\|_1}(t)$ of power type $p- \varepsilon$.  % (see also \cite{KaltonRandrianarivony2008}). 
Similarly, for any $\varepsilon_1 >0$, $X^*$ has an equivalent norm $\|.\|_2^*$ with $\ov{\rho}_{\|.\|_2^*}(t)$ of power type $p' - \varepsilon_1$. By duality \cite{Prus1989} (see also \cite{GodefroyKaltonLancien2001}), $X$ has $\ov{\delta}_{\|.\|_2}(t)$ of power type $p +
\varepsilon$ provided we choose $\varepsilon_1$ accordingly. Then by \cite[Proposition 2.10]{JohnsonLindenstraussPreissSchechtman2002} (see also \cite{OdellSchlumprecht2006}), $X$ admits an equivalent norm $\|.\|$ with both $\ov{\rho}(t)$ of power type $p- \varepsilon$ and
$\ov{\delta}(t)$ of power type $p+ \varepsilon$.

It was shown in \cite{Kutzarova1990} that if a norm is both NUC and NUS then it also has property $(\beta)$. Following the same avenue, it was proved in \cite{DilworthKutzarovaRandrianarivonyRevalskiZhivkov2012} that a norm with the above power types for $\ov{\rho}(t)$ and $\ov{\delta}(t)$ has a $(\beta)$-modulus $\ov{\beta}(t)$ of power type
$$
\frac{(p-\varepsilon)(p+\varepsilon)-(p-\varepsilon)}{(p-\varepsilon)-1}=
p+o(\eps).
$$
%Letting $\varepsilon$ tend to $0$, $\varepsilon'$ can be made as close to $0$ as we want.
\end{proof}

\begin{cor}
For every  $\varepsilon >0$ the space $K_p =\left(\left( \sum_{n=1}^\infty G_n\right)_{T_p}\right)^2$, $1<p<\infty$, has an equivalent norm with $(\beta)$-modulus of power type $p+ \varepsilon$.
\end{cor}
\begin{proof}
By \cite{Kalton2010}, $K_p$ is uniformly homeomorphic to $C_p$ and $\ov{\rho}_{_{C_p}}(t) \le C_1t^p$. On the other hand, we have $
{K_p}^* = \left(( \sum_{n=1}^{\infty} G_n^*)_{T_p^*}\right)^2 = \left(( \sum_{n=1}^\infty G_n^*)_{T_{p'}}\right)^2$.  Since $(G_n^*)_{n\geq 1}$ is also dense in Banach-Mazur metric, we obtain by the same theorem in \cite{Kalton2010} that $\left(\left( \sum_{n=1}^{\infty} G_n^*\right)_{T_{p'}}\right)^2$ is uniformly homeomorphic to $C_{p'}$; and $\ov{\rho}_{_{C_{p'}}}(t) \le C_2t^{p'}$.
\end{proof}

\section{Acknowledgments}
%\begin{ack*}
This work was initiated while the authors attended the NSF funded Workshop in Linear Analysis and Probability at Texas A\&M University during Summer 2012, and was finalized while the first two authors attended this same Workshop during Summer 2013.

The authors would like to thank the anonymous referee for his/her invaluable comments that helped smooth out the presentation of this paper.

%\end{ack*}

\end{document}